\documentclass[11pt]{article}
\usepackage[margin=1in]{geometry}
\usepackage[T1]{fontenc}
\usepackage[utf8]{inputenc}
\usepackage{lmodern}
\usepackage{microtype}
\usepackage{amsmath,amssymb,amsthm,mathtools}
\usepackage{enumitem}
\usepackage{bm}
\usepackage{xcolor}
\usepackage{hyperref}
\usepackage[nameinlink,capitalise,noabbrev]{cleveref}
\usepackage{aliascnt}

\hypersetup{
  colorlinks=true,
  linkcolor=blue!50!black,
  citecolor=blue!50!black,
  urlcolor=blue!50!black,
  pdftitle={Exact finite-horizon quantile Kelly for repeated multi-outcome events},
  pdfauthor={Christopher D. Long}
}

\newtheorem{theorem}{Theorem}[section]
\newaliascnt{proposition}{theorem}
\newtheorem{proposition}[proposition]{Proposition}
\aliascntresetthe{proposition}
\newaliascnt{corollary}{theorem}
\newtheorem{corollary}[corollary]{Corollary}
\aliascntresetthe{corollary}
\newaliascnt{lemma}{theorem}
\newtheorem{lemma}[lemma]{Lemma}
\aliascntresetthe{lemma}
\theoremstyle{definition}
\newaliascnt{definition}{theorem}
\newtheorem{definition}[definition]{Definition}
\aliascntresetthe{definition}
\newaliascnt{remark}{theorem}
\newtheorem{remark}[remark]{Remark}
\aliascntresetthe{remark}
\newaliascnt{example}{theorem}
\newtheorem{example}[example]{Example}
\aliascntresetthe{example}
\newaliascnt{conjecture}{theorem}
\newtheorem{conjecture}[conjecture]{Conjecture}
\aliascntresetthe{conjecture}

\crefname{theorem}{theorem}{theorems}
\Crefname{theorem}{Theorem}{Theorems}
\crefname{proposition}{proposition}{propositions}
\Crefname{proposition}{Proposition}{Propositions}
\crefname{corollary}{corollary}{corollaries}
\Crefname{corollary}{Corollary}{Corollaries}
\crefname{lemma}{lemma}{lemmas}
\Crefname{lemma}{Lemma}{Lemmas}
\crefname{definition}{definition}{definitions}
\Crefname{definition}{Definition}{Definitions}
\crefname{remark}{remark}{remarks}
\Crefname{remark}{Remark}{Remarks}
\crefname{example}{example}{examples}
\Crefname{example}{Example}{Examples}
\crefname{conjecture}{conjecture}{conjectures}
\Crefname{conjecture}{Conjecture}{Conjectures}

\newcommand{\R}{\mathbb{R}}
\newcommand{\N}{\mathbb{N}}
\newcommand{\E}{\mathbb{E}}
\newcommand{\Prob}{\mathbb{P}}

\newcommand{\supp}{\operatorname{supp}}
\newcommand{\Med}{\operatorname{Med}}
\newcommand{\Quant}{\operatorname{Q}}

\title{Exact Finite-Horizon Quantile Kelly for Repeated Multi-Outcome Events}
\author{Christopher D. Long\\Headlamp Software\\\texttt{galizur@gmail.com}}
\date{}

\begin{document}
\maketitle

\begin{abstract}
We formulate and prove an exact finite-horizon quantile theorem for repeated identical multi-outcome Kelly wagering in wealth-profile / Arrow--Debreu coordinates. For a fixed $m$-outcome event repeated independently over a horizon $n$, the terminal wealth induced by a one-period wealth profile $W$ is a monomial $W^N$ in the multinomial count vector $N$. We show that every fixed upper quantile of terminal wealth is a positively homogeneous piecewise-monomial function on the closed Arrow--Debreu wealth simplex, equivalently piecewise linear in log-wealth coordinates on the positive interior. The pieces are indexed by the chambers of the multinomial count arrangement, and on each chamber the quantile objective is exactly a one-period Kelly objective for a count-based \emph{shadow law} $k/n$. Consequently the finite-horizon quantile problem decomposes into finitely many shadow-Kelly subproblems.

We then refine the interior chamber picture to a finite stratification of the full closed simplex by support faces and arrangement faces, and we prove a weak exact recursive boundary algorithm. We also prove a natural first-order asymptotic collapse to ordinary Kelly, showing that the optimal scaled log-quantile converges to the ordinary Kelly value and that exact finite-horizon maximizers converge to the Kelly wealth profile. For illustration, we include worked binary and ternary examples in the main text and expanded versions in the appendices. We conclude with further remarks and conjectures concerning stronger pruning, higher-order finite-horizon corrections, and extensions to simultaneous wagers.
\end{abstract}

\section{Introduction}

The Kelly criterion, introduced in \cite{Kelly1956}, selects bets by maximizing expected log wealth. In finite-state favorable games, the asymptotic dominance of log-optimal strategies was placed on rigorous footing by Breiman \cite{Breiman1961}, and in more general sequential investment settings by Algoet and Cover \cite{AlgoetCover1988}. For the median of fortune, Ethier \cite{Ethier2004} proved an exact repeated binary finite-horizon result and established asymptotic median-optimality more broadly. Broader background on the capital-growth criterion, together with its gambling and investment applications, may be found in \cite{MacLeanThorpZiemba2011,Thorp2008}. The purpose of the present note is different: we work entirely in the wealth-profile / Arrow--Debreu language and show that for a fixed multi-outcome event repeated independently over a finite horizon, the exact finite-horizon quantile problem has a clean chamber decomposition into shadow-Kelly subproblems.

The key point is that, once one uses state-contingent wealth rather than raw stake fractions, repeated play becomes algebraic. If the one-period wealth profile is $W=(W_1,\dots,W_m)$ and the $m$-outcome event is repeated $n$ times, then terminal wealth is
\[
X_n(W)=\prod_{i=1}^m W_i^{N_i},
\]
where $N=(N_1,\dots,N_m)$ is the multinomial count vector. Thus every possible terminal wealth is indexed by a count vector $k\in\mathcal C_n$, and every fixed upper quantile is determined by the ordering of the finitely many monomials $W^k$. Passing to log-wealth coordinates $u=\log W$ shows that this ordering is controlled by the linear forms $k\cdot u$. The hyperplanes $(k-\ell)\cdot u=0$ therefore partition $u$-space into chambers on which the quantile objective is linear. The corresponding optimization problem on each chamber is precisely a one-period Kelly problem for the empirical shadow law $k/n$.

The resulting theorem is exact, finite-horizon, and entirely structural. It is not an asymptotic approximation and does not invoke a central limit theorem, a law of large numbers, or any large-deviation principle. Probability enters only through the multinomial masses attached to count vectors. In the framework language: finite-horizon quantile Kelly is \emph{piecewise shadow-Kelly}. After refining the chamber picture to all support faces, the same geometric machine also yields a weak exact recursive boundary algorithm on the full closed simplex. The local stratum objective also has a natural information-geometric reading: in suitable ratio coordinates it becomes a softmax log-likelihood in natural parameters, so the strict concavity used later is the familiar exponential-family concavity. We then show that the resulting exact finite-horizon quantile problem collapses, at first asymptotic order, to the ordinary one-period Kelly problem.

\paragraph{Position within the framework.}
The theorem is a direct deployment of the core wealth-profile / Arrow--Debreu viewpoint developed in \cite{Long2026Single,Long2026Simultaneous,Long2026Parlay}. The central ideas are these.
\begin{enumerate}[label=(\roman*), itemsep=2pt]
  \item A one-period wager is a \emph{wealth profile} $W$, not a stake vector.
  \item Feasibility is expressed by the Arrow--Debreu budget constraint $q\cdot W=1$.
  \item Once the active count vector is fixed, the resulting finite-horizon objective is a one-period Kelly objective for a count-based \emph{shadow law} $k/n$.
  \item Boundary behavior is organized by a finite stratification of the closed simplex into support faces and arrangement faces.
\end{enumerate}
The present note is self-contained and can be read independently, but the background framework language of implicit state positions, shadow-Kelly reduction, and eventwise factorization is developed in \cite{Long2026Single,Long2026Simultaneous,Long2026Parlay}.
For illustration, we include a short worked binary example and a short worked ternary example in the body of the paper, with more detailed companion analyses in the appendices.

\section{Framework, notation, and problem statement}

Fix an integer $m\ge 2$. Let
\[
p=(p_1,\dots,p_m)\in(0,1)^m,\qquad \sum_{i=1}^m p_i=1,
\]
be the true law of a single $m$-outcome event, and let
\[
q=(q_1,\dots,q_m)\in(0,\infty)^m
\]
be a strictly positive Arrow--Debreu state-price vector. We do not assume $\sum_i q_i=1$; the only essential assumption is strict positivity.

\begin{definition}[Closed and open Arrow--Debreu wealth simplices]
The \emph{closed} and \emph{open} one-period wealth simplices are
\[
\overline{\mathcal W}:=\{W\in[0,\infty)^m:\ q\cdot W=1\},
\qquad
\mathcal W^{++}:=\{W\in(0,\infty)^m:\ q\cdot W=1\}.
\]
A point $W\in\overline{\mathcal W}$ is a feasible one-period wealth profile. When $W\in\mathcal W^{++}$ we write $u=\log W\in\R^m$ coordinatewise.
\end{definition}

We now repeat the \emph{same} one-period event independently over a horizon $n\in\N$. Let
\[
N=(N_1,\dots,N_m)\sim \mathrm{Multinomial}(n,p),
\]
so that $N_i$ counts how often outcome $i$ occurs. The terminal wealth induced by the fixed one-period profile $W\in\overline{\mathcal W}$ is
\begin{equation}\label{eq:terminal-wealth}
X_n(W):=\prod_{i=1}^m W_i^{N_i},
\end{equation}
with the convention $0^0=1$.

The relevant finite set of count vectors is
\[
\mathcal C_n:=\Bigl\{k=(k_1,\dots,k_m)\in\mathbb N_0^m:\ \sum_{i=1}^m k_i=n\Bigr\}.
\]
For $k\in\mathcal C_n$, define
\[
\pi_k:=\Prob(N=k)=\frac{n!}{k_1!\cdots k_m!}\prod_{i=1}^m p_i^{k_i},
\qquad
W^k:=\prod_{i=1}^m W_i^{k_i}.
\]
Then $X_n(W)$ takes the values $W^k$, $k\in\mathcal C_n$, with probabilities $\pi_k$.

\begin{definition}[Upper quantile]
For a nonnegative random variable $Y$ and a level $\alpha\in(0,1)$, define the \emph{upper $\alpha$-quantile}
\[
\Quant_\alpha^+(Y):=\sup\{t\ge 0:\ \Prob(Y\ge t)\ge \alpha\}.
\]
The upper median is the special case $\Med^+(Y):=\Quant_{1/2}^+(Y)$.
\end{definition}

The finite-horizon optimization problem studied in this note is
\begin{equation}\label{eq:master-problem}
\sup_{W\in\overline{\mathcal W}} \Quant_\alpha^+\bigl(X_n(W)\bigr),
\end{equation}
with particular emphasis on the median case $\alpha=1/2$.

\section{The chamber theorem}

\subsection{The count arrangement}

The ordering of the monomials $W^k$ is easiest to express on the positive simplex in log-wealth coordinates. For $k,\ell\in\mathcal C_n$, define the central hyperplane
\[
H_{k,\ell}:=\{u\in\R^m:(k-\ell)\cdot u=0\}.
\]
The finite arrangement
\[
\mathcal H_n:=\{H_{k,\ell}:k,\ell\in\mathcal C_n,\ k\ne \ell\}
\]
partitions $\R^m$ into finitely many open polyhedral cones (chambers) and lower-dimensional faces.

\begin{definition}[Admissible chamber]
An open chamber $\Gamma$ of $\mathcal H_n$ is \emph{admissible} if
\[
\mathcal F_\Gamma:=\{W\in\mathcal W^{++}:\ \log W\in\Gamma\}
\]
is nonempty. We write $\overline{\mathcal F_\Gamma}$ for its closure in $\overline{\mathcal W}$.
\end{definition}

On an admissible chamber $\Gamma$, the quantities $k\cdot u$ are strictly ordered and therefore so are the monomials $W^k=e^{k\cdot u}$.

\begin{lemma}[Chamber ordering]\label{lem:chamber-ordering}
Fix an admissible chamber $\Gamma$. Then there exists a unique ordering
\[
k^{(1)}_\Gamma,\dots,k^{(L)}_\Gamma,
\qquad L:=|\mathcal C_n|,
\]
of the count vectors in $\mathcal C_n$ such that for every $W\in\mathcal F_\Gamma$,
\[
W^{k^{(1)}_\Gamma}>W^{k^{(2)}_\Gamma}>\cdots>W^{k^{(L)}_\Gamma}.
\]
Moreover, for every $W\in\overline{\mathcal F_\Gamma}$,
\[
W^{k^{(1)}_\Gamma}\ge W^{k^{(2)}_\Gamma}\ge\cdots\ge W^{k^{(L)}_\Gamma}.
\]
\end{lemma}

\begin{proof}
Strict ordering on $\mathcal F_\Gamma$ is immediate from the definition of a chamber: if $u=\log W\in\Gamma$, then for every distinct $k,\ell\in\mathcal C_n$ one has either $(k-\ell)\cdot u>0$ or $(k-\ell)\cdot u<0$, and the sign is constant on $\Gamma$. Hence the ordering of the finitely many numbers $k\cdot u$ is constant on $\Gamma$, which yields the unique list $k^{(1)}_\Gamma,\dots,k^{(L)}_\Gamma$.

Now let $W\in\overline{\mathcal F_\Gamma}$. There exists a sequence $W^{(r)}\in\mathcal F_\Gamma$ with $W^{(r)}\to W$ in $\overline{\mathcal W}$. For every $r$,
\[
\bigl(W^{(r)}\bigr)^{k^{(1)}_\Gamma}>\cdots>\bigl(W^{(r)}\bigr)^{k^{(L)}_\Gamma}.
\]
Passing to the limit and using continuity of $W\mapsto W^k$ on $\overline{\mathcal W}$ gives the weak inequalities on $\overline{\mathcal F_\Gamma}$.
\end{proof}

\begin{definition}[Quantile index on a chamber]
Fix $\alpha\in(0,1)$ and an admissible chamber $\Gamma$. Let $r_\alpha(\Gamma)$ be the smallest index such that
\[
\sum_{j=1}^{r_\alpha(\Gamma)} \pi_{k^{(j)}_\Gamma}\ge \alpha.
\]
Define the corresponding \emph{quantile count vector}
\[
k_{\alpha,\Gamma}:=k^{(r_\alpha(\Gamma))}_\Gamma.
\]
\end{definition}

We can now state the exact chamber theorem.

\begin{theorem}[Exact chamber quantile theorem]\label{thm:chamber-quantile}
Fix $\alpha\in(0,1)$.
\begin{enumerate}[label=(\alph*), itemsep=4pt]
  \item The closed simplex $\overline{\mathcal W}$ is covered by the finitely many closures of admissible chambers:
  \[
  \overline{\mathcal W}=\bigcup_{\Gamma}\overline{\mathcal F_\Gamma}.
  \]
  \item For every admissible chamber $\Gamma$ and every $W\in\overline{\mathcal F_\Gamma}$,
  \begin{equation}\label{eq:chamber-quantile}
  \Quant_\alpha^+\bigl(X_n(W)\bigr)=W^{k_{\alpha,\Gamma}}.
  \end{equation}
  In particular, on $\mathcal F_\Gamma$ one has
  \[
  \log \Quant_\alpha^+\bigl(X_n(W)\bigr)=k_{\alpha,\Gamma}\cdot \log W.
  \]
  \item Consequently,
  \begin{equation}\label{eq:finite-decomp}
  \sup_{W\in\overline{\mathcal W}} \Quant_\alpha^+\bigl(X_n(W)\bigr)
  =
  \max_{\Gamma}\ \sup_{W\in\overline{\mathcal F_\Gamma}} W^{k_{\alpha,\Gamma}}.
  \end{equation}
\end{enumerate}
\end{theorem}

\begin{proof}
(a) Let $W\in\overline{\mathcal W}$. Choose a sequence $W^{(r)}\in\mathcal W^{++}$ with $W^{(r)}\to W$. Since there are only finitely many admissible chambers, some chamber $\Gamma$ contains infinitely many of the points $\log W^{(r)}$. Passing to that subsequence gives $W\in\overline{\mathcal F_\Gamma}$. Hence the closures cover $\overline{\mathcal W}$.

(b) Fix an admissible chamber $\Gamma$ and $W\in\overline{\mathcal F_\Gamma}$. By \cref{lem:chamber-ordering},
\[
W^{k^{(1)}_\Gamma}\ge W^{k^{(2)}_\Gamma}\ge\cdots\ge W^{k^{(L)}_\Gamma}.
\]
Write $r=r_\alpha(\Gamma)$. Then by definition,
\[
\sum_{j=1}^{r}\pi_{k^{(j)}_\Gamma}\ge \alpha,
\qquad
\sum_{j=1}^{r-1}\pi_{k^{(j)}_\Gamma}<\alpha.
\]
Since $X_n(W)$ takes the value $W^{k^{(j)}_\Gamma}$ with probability $\pi_{k^{(j)}_\Gamma}$,
\[
\Prob\!\left(X_n(W)\ge W^{k^{(r)}_\Gamma}\right)
\ge
\sum_{j=1}^{r}\pi_{k^{(j)}_\Gamma}
\ge \alpha.
\]
On the other hand, if $t>W^{k^{(r)}_\Gamma}$, then only those indices $j$ with $W^{k^{(j)}_\Gamma}>W^{k^{(r)}_\Gamma}$ can contribute to the event $\{X_n(W)\ge t\}$. By the weak ordering above, every such index satisfies $j\le r-1$. Hence
\[
\Prob\bigl(X_n(W)\ge t\bigr)
\le
\sum_{j=1}^{r-1}\pi_{k^{(j)}_\Gamma}
<\alpha.
\]
Therefore \eqref{eq:chamber-quantile} holds.

(c) This follows immediately from (a) and (b).
\end{proof}

\begin{corollary}[Median case]\label{cor:median-case}
For $\alpha=1/2$, the upper median of terminal wealth is a piecewise monomial function on $\overline{\mathcal W}$ and a piecewise linear function of $u=\log W$ on $\mathcal W^{++}$.
\end{corollary}

\begin{remark}[What the chamber theorem says]
The theorem is exact and finite-horizon. It says that the entire finite-horizon quantile problem is controlled by a finite hyperplane arrangement in log-wealth space. The count vector $k_{\alpha,\Gamma}$ plays the role of an \emph{active shadow law} on the chamber $\Gamma$. In the median case, finite-horizon median Kelly is therefore piecewise shadow-Kelly.
\end{remark}

\section{Shadow-Kelly reduction}

The chamber theorem reduces the global quantile problem to finitely many subproblems of the form
\[
\sup_{W\in\overline{\mathcal F_\Gamma}} W^k.
\]
We now identify the unconstrained optimizer of $W^k$ on the full Arrow--Debreu simplex.

\begin{proposition}[Shadow-Kelly maximizer of a count monomial]\label{prop:monomial-max}
Fix $k\in\mathcal C_n$ and define the empirical law
\[
\widehat p^{(k)}:=\frac{k}{n}\in\Delta^{m-1}.
\]
Then the unique maximizer of $W^k$ over $\overline{\mathcal W}$ is
\begin{equation}\label{eq:shadow-kelly-candidate}
W^{(k)}_i=\frac{k_i}{nq_i}=\frac{\widehat p^{(k)}_i}{q_i},
\qquad i=1,\dots,m,
\end{equation}
with the convention $W^{(k)}_i=0$ when $k_i=0$. Its optimal value is
\begin{equation}\label{eq:monomial-opt-value}
\max_{W\in\overline{\mathcal W}} W^k
=
\prod_{i=1}^m\left(\frac{k_i}{nq_i}\right)^{k_i}.
\end{equation}
Equivalently, $W^{(k)}$ is the one-period Kelly optimizer for the shadow law $\widehat p^{(k)}$.
\end{proposition}

\begin{proof}
Let
\[
S:=\{i: k_i>0\}.
\]
For every $W\in\overline{\mathcal W}$,
\[
\sum_{i\in S} q_iW_i \le \sum_{i=1}^m q_iW_i=1.
\]
Apply the weighted arithmetic--geometric mean inequality with weights $k_i/n$ on the positive numbers $nq_iW_i/k_i$, $i\in S$:
\[
\prod_{i\in S}\left(\frac{nq_iW_i}{k_i}\right)^{k_i/n}
\le
\sum_{i\in S}\frac{k_i}{n}\,\frac{nq_iW_i}{k_i}
=
\sum_{i\in S} q_iW_i
\le 1.
\]
Raise both sides to the $n$th power to obtain
\[
\prod_{i\in S}\left(\frac{nq_iW_i}{k_i}\right)^{k_i}\le 1.
\]
Rearranging gives
\[
W^k=\prod_{i=1}^m W_i^{k_i}
\le
\prod_{i\in S}\left(\frac{k_i}{nq_i}\right)^{k_i},
\]
which is exactly \eqref{eq:monomial-opt-value}.

Equality holds in weighted AM--GM iff all the quantities $nq_iW_i/k_i$ ($i\in S$) are equal, and equality in $\sum_{i\in S}q_iW_i\le 1$ forces $W_i=0$ for $i\notin S$. Since $\sum_i q_iW_i=1$, the common value must be $1$, so necessarily
\[
q_iW_i=\frac{k_i}{n}\quad (i\in S),
\qquad
W_i=0\quad (i\notin S).
\]
This is exactly \eqref{eq:shadow-kelly-candidate}. Uniqueness follows from the equality condition.
\end{proof}

\begin{remark}[Why this is shadow Kelly]
The monomial objective can be written as
\[
W^k=\exp\!\left( n\sum_{i=1}^m \widehat p^{(k)}_i \log W_i\right).
\]
Thus maximizing $W^k$ is equivalent to maximizing expected one-period log wealth under the shadow empirical law $\widehat p^{(k)}=k/n$. In the language of the framework, \cref{prop:monomial-max} says: once the active count vector is fixed, the finite-horizon problem becomes an ordinary one-period Kelly problem for a shadow law.
\end{remark}

Combining \cref{thm:chamber-quantile,prop:monomial-max} yields the main structural theorem.

\begin{theorem}[Exact finite-horizon quantile Kelly decomposition]\label{thm:main-structural}
Fix $\alpha\in(0,1)$. Then the finite-horizon quantile problem \eqref{eq:master-problem} decomposes into finitely many shadow-Kelly subproblems:
\[
\sup_{W\in\overline{\mathcal W}} \Quant_\alpha^+\bigl(X_n(W)\bigr)
=
\max_{\Gamma}\ \sup_{W\in\overline{\mathcal F_\Gamma}} W^{k_{\alpha,\Gamma}}.
\]
For each admissible chamber $\Gamma$:
\begin{enumerate}[label=(\roman*), itemsep=3pt]
  \item if the shadow-Kelly point $W^{(k_{\alpha,\Gamma})}$ belongs to $\overline{\mathcal F_\Gamma}$, then it is the exact optimizer of the chamber subproblem;
  \item otherwise the chamber optimum lies on the boundary $\partial\overline{\mathcal F_\Gamma}$.
\end{enumerate}
In particular, finite-horizon quantile Kelly is \emph{piecewise shadow-Kelly}.
\end{theorem}

\begin{proof}
The decomposition is \cref{thm:chamber-quantile}(c). For a fixed chamber $\Gamma$, the chamber objective is $W\mapsto W^{k_{\alpha,\Gamma}}$. By \cref{prop:monomial-max}, its unique global maximizer over all of $\overline{\mathcal W}$ is $W^{(k_{\alpha,\Gamma})}$. If this point lies in the feasible chamber closure $\overline{\mathcal F_\Gamma}$, then it solves the chamber subproblem. If not, then the maximizer of the continuous function $W\mapsto W^{k_{\alpha,\Gamma}}$ over the compact set $\overline{\mathcal F_\Gamma}$ cannot lie in its relative interior (otherwise it would also be the global maximizer), so it lies on the boundary.
\end{proof}

\begin{remark}[Boundary refinement]
\Cref{thm:main-structural} gives the correct exact structural stopping point at the level of chambers. \Cref{sec:support-strata,sec:weak-recursion} refine this boundary picture: the closed simplex is stratified by support faces and arrangement faces, and the resulting finite boundary descent becomes an exact, though not yet efficiently pruned, recursive solver.
\end{remark}

\begin{proposition}[Restriction to incomplete one-period wealth families]\label{prop:restricted-family}
Let $\mathcal W_{\mathrm{feas}}\subseteq\overline{\mathcal W}$ be any nonempty closed feasible family. Then
\[
\sup_{W\in\mathcal W_{\mathrm{feas}}} \Quant_\alpha^+\bigl(X_n(W)\bigr)
=
\max_{\Gamma}\ \sup_{W\in\mathcal W_{\mathrm{feas}}\cap\overline{\mathcal F_\Gamma}} W^{k_{\alpha,\Gamma}}.
\]
Thus the exact chamber theorem survives unchanged under restriction to an arbitrary closed one-period wealth family.
\end{proposition}

\begin{proof}
The proof is identical to the proof of \cref{thm:chamber-quantile}(c), now restricted to $\mathcal W_{\mathrm{feas}}$.
\end{proof}

\begin{remark}[Why \cref{prop:restricted-family} matters]
This proposition is the exact bridge from the complete-market theorem to incomplete one-period menus. In particular, if one-period wealth profiles arise from a smaller family of available wagers or structured simultaneous bets, then the finite-horizon quantile functional is simply the restriction of the same piecewise-monomial function to the feasible wealth family. When the feasible family is cut out by a one-period risk constraint, each chamber or stratum subproblem becomes a risk-constrained shadow-Kelly problem of the same type studied in \cite{Long2026Risk}; thus the present finite-horizon theorem composes directly with that one-period theory.
\end{remark}

\section{Binary specialization and exact repeated two-outcome Kelly}

The repeated binary case is recovered transparently from the chamber theorem.

\begin{corollary}[Repeated exact binary quantile Kelly]\label{cor:binary}
Assume $m=2$, and write a one-period wealth profile as $W=(U,D)\in\overline{\mathcal W}$, where $q_1U+q_2D=1$. Let $B_n\sim\mathrm{Binomial}(n,p_1)$ be the number of occurrences of outcome $1$.
\begin{enumerate}[label=(\alph*), itemsep=3pt]
  \item On the chamber $U>D$, terminal wealth is increasing in $B_n$, and
  \[
  \Quant_\alpha^+\bigl(X_n(U,D)\bigr)=U^{m_\alpha}(D)^{n-m_\alpha},
  \]
  where $m_\alpha$ is the largest integer $j$ with
  \[
  \Prob(B_n\ge j)\ge \alpha.
  \]
  \item On that chamber, the unconstrained chamber optimizer is the shadow-Kelly point
  \[
  U^*=\frac{m_\alpha}{nq_1},
  \qquad
  D^*=\frac{n-m_\alpha}{nq_2}.
  \]
  \item For $\alpha=1/2$, this recovers the exact repeated binary upper-median theorem: finite-horizon exact median Kelly is the one-period Kelly optimizer for the shadow probability $m_{1/2}/n$.
\end{enumerate}
\end{corollary}

\begin{proof}
There are only two open chambers, $U>D$ and $U<D$, separated by the wall $U=D$. On $U>D$, the monomials $U^jD^{n-j}$ are strictly ordered by $j$, so the chamber ordering is determined by the integer $j=B_n$. The quantile index is therefore the largest $j$ with tail probability at least $\alpha$. The shadow-Kelly optimizer is then an immediate specialization of \cref{prop:monomial-max}. The median case is $\alpha=1/2$.
\end{proof}

\begin{remark}[Comparison with Ethier]
Ethier \cite{Ethier2004} proved the repeated binary exact median theorem directly. \Cref{cor:binary} shows that the binary result is the one-dimensional shadow of the general chamber theorem. The binary shadow probability is simply the median count law of the relevant chamber.
\end{remark}

\begin{example}[A concrete binary finite-horizon median problem]\label{ex:binary-main}
Let
\[
p=(0.6,0.4),\qquad q=(1,1),
\]
so the one-period wealth simplex is
\[
\overline{\mathcal W}=\{(U,D)\in[0,\infty)^2: U+D=1\}.
\]
Take horizon $n=3$ and upper median level $\alpha=\tfrac12$. The four count vectors in $\mathcal C_3$, together with the corresponding monomials and probabilities, are
\[
\begin{array}{c|c|c}
k & W^k & \pi_k \\ \hline
(3,0) & U^3 & 0.216 \\
(2,1) & U^2D & 0.432 \\
(1,2) & UD^2 & 0.288 \\
(0,3) & D^3 & 0.064
\end{array}
\]
On the chamber $\Gamma_+:=\{U>D\}$ the monomials are ordered
\[
U^3>U^2D>UD^2>D^3,
\]
so the upper median is reached at the second term:
\[
0.216<\tfrac12<0.216+0.432=0.648.
\]
Hence
\[
\Med^+\bigl(X_3(U,D)\bigr)=U^2D
\qquad (U>D),
\]
with active count vector $(2,1)$ and shadow-Kelly point
\[
\Bigl(\frac23,\frac13\Bigr).
\]
Since $\frac23>\frac13$, this point lies in $\Gamma_+$ and therefore solves the chamber subproblem exactly.

On the opposite chamber $\Gamma_-:=\{U<D\}$, the ordering reverses:
\[
D^3>UD^2>U^2D>U^3.
\]
Now
\[
0.064+0.288=0.352<\tfrac12<0.352+0.432=0.784,
\]
so the upper median is again $U^2D$, now appearing as the third term in the chamber ordering, with the same shadow-Kelly point $\bigl(\frac23,\frac13\bigr)$. Now, however, that point lies outside $\Gamma_-$, so the chamber optimum is forced to the wall $U=D$. Since $U+D=1$, this gives
\[
(U,D)=\Bigl(\frac12,\frac12\Bigr),
\qquad
\Med^+\bigl(X_3(U,D)\bigr)=\frac18.
\]
Thus the global optimum is attained on $\Gamma_+$ at
\[
\Bigl(\frac23,\frac13\Bigr),
\qquad
\sup_{W\in\overline{\mathcal W}}\Med^+\bigl(X_3(W)\bigr)=\frac{4}{27}.
\]
For comparison, the ordinary one-period Kelly point is
\[
W^{\mathrm K}=(0.6,0.4),
\]
so the exact finite-horizon median optimizer differs from the asymptotic Kelly profile even in this smallest nontrivial binary example. A slightly more detailed chamber-by-chamber discussion is given in \cref{app:binary-example}.
\end{example}

\section{Support-face refinement and strata}\label{sec:support-strata}

The chamber theorem describes the positive interior $\mathcal W^{++}$ and its chamber closures. To analyze exact boundary descent, one needs a direct structural description of the whole closed simplex.

\begin{definition}[Open and closed support faces]
For each nonempty subset $S\subseteq\{1,\dots,m\}$, define the associated open and closed support faces by
\[
\mathcal W_S^\circ:=\{W\in\overline{\mathcal W}: W_i>0 \iff i\in S\},
\qquad
\overline{\mathcal W}_S:=\{W\in\overline{\mathcal W}: W_i=0 \text{ for } i\notin S\}.
\]
Thus $\mathcal W_S^\circ$ is the relative interior of the face $\overline{\mathcal W}_S$.
\end{definition}

\begin{remark}[Support decomposition]
The closed simplex admits the finite disjoint decomposition
\[
\overline{\mathcal W}=\bigsqcup_{\varnothing\neq S\subseteq\{1,\dots,m\}} \mathcal W_S^\circ.
\]
If $W\in\overline{\mathcal W}_S$ and $\supp(k)\not\subseteq S$, then $W^k=0$.
\end{remark}

Fix a nonempty support set $S$, and choose once and for all an anchor index $r=r(S)\in S$. On the open support face $\mathcal W_S^\circ$, introduce ratio coordinates
\[
z_i:=\log\frac{W_i}{W_r},
\qquad i\in S\setminus\{r\}.
\]
These coordinates identify $\mathcal W_S^\circ$ with $\R^{S\setminus\{r\}}$, since one recovers $W$ from $z$ via
\begin{equation}\label{eq:ratio-param}
W_r(z)=\Bigl(q_r+\sum_{i\in S\setminus\{r\}} q_i e^{z_i}\Bigr)^{-1},
\qquad
W_i(z)=e^{z_i}W_r(z)
\quad (i\in S\setminus\{r\}).
\end{equation}
The induced hyperplane arrangement and face stratification on $\mathcal W_S^\circ$ are independent of the anchor choice; only the coordinate description changes.

\begin{lemma}[Monomial comparisons are linear in ratio coordinates]\label{lem:ratio-linear}
Fix a support set $S$ and $k,\ell\in\mathcal C_n$ with $\supp(k),\supp(\ell)\subseteq S$. Then on $\mathcal W_S^\circ$,
\[
\log\frac{W^k}{W^\ell}
=
\sum_{i\in S\setminus\{r\}} (k_i-\ell_i) z_i.
\]
In particular, the comparison $W^k\gtrless W^\ell$ is determined by the sign of a linear form in the ratio coordinates.
\end{lemma}

\begin{proof}
Since $W_i=e^{z_i}W_r$ for $i\in S\setminus\{r\}$ and both count vectors satisfy $\sum_i k_i=\sum_i \ell_i=n$, one has
\[
W^k=W_r^n \prod_{i\in S\setminus\{r\}} e^{k_i z_i},
\qquad
W^\ell=W_r^n \prod_{i\in S\setminus\{r\}} e^{\ell_i z_i}.
\]
Taking logarithms of the ratio gives the formula.
\end{proof}

\begin{definition}[Arrangement faces on a support face]
Fix a nonempty support set $S$. For each pair $k,\ell\in\mathcal C_n$ with $\supp(k),\supp(\ell)\subseteq S$, define the hyperplane
\[
H_{k,\ell}^S
:=
\Bigl\{z\in\R^{S\setminus\{r\}}:
\sum_{i\in S\setminus\{r\}} (k_i-\ell_i)z_i=0\Bigr\}.
\]
Let $\mathcal H_S$ denote the finite arrangement of all such hyperplanes, and let $\mathfrak F_S$ denote the set of all relatively open faces of that arrangement.
\end{definition}

\begin{definition}[Strata]
For each nonempty support set $S$ and each face $F\in\mathfrak F_S$, define the corresponding stratum
\[
\sigma(S,F):=\{W(z): z\in F\}\subseteq \mathcal W_S^\circ,
\]
where $W(z)$ is given by \eqref{eq:ratio-param}.
\end{definition}

\begin{lemma}[Finite stratification]\label{lem:finite-stratification}
The family of strata $\sigma(S,F)$ is finite, pairwise disjoint, and its union is all of $\overline{\mathcal W}$.
\end{lemma}

\begin{proof}
There are only finitely many nonempty support sets $S$, and for each fixed $S$ the arrangement $\mathcal H_S$ has only finitely many relatively open faces. Hence the family is finite.

Fix $W\in\overline{\mathcal W}$. Let $S:=\{i: W_i>0\}$. Then $W\in\mathcal W_S^\circ$. In the ratio coordinates on $\mathcal W_S^\circ$, the corresponding point $z$ lies in a unique relatively open face $F$ of the arrangement $\mathcal H_S$. Thus $W\in\sigma(S,F)$. This proves that the strata cover $\overline{\mathcal W}$. Pairwise disjointness is immediate from uniqueness of the support set and of the arrangement face.
\end{proof}

\section{Weak exact recursive boundary algorithm}\label{sec:weak-recursion}

We now refine the interior chamber theorem to a finite exact recursive solver on the full closed simplex.

\begin{theorem}[Exact quantile description on a stratum]\label{thm:stratum-quantile}
Fix $\alpha\in(0,1)$ and a stratum $\sigma=\sigma(S,F)$. Write
\[
\Pi_S:=\sum_{\substack{k\in\mathcal C_n\\ \supp(k)\subseteq S}} \pi_k
=
\Prob(N_i=0\text{ for every } i\notin S).
\]
Then exactly one of the following alternatives holds.
\begin{enumerate}[label=(\alph*), itemsep=3pt]
  \item If $\Pi_S<\alpha$, then
  \[
  \Quant_\alpha^+\bigl(X_n(W)\bigr)=0
  \qquad\text{for every } W\in\sigma.
  \]
  \item If $\Pi_S\ge \alpha$, then there exists a count vector $k_\sigma\in\mathcal C_n$ with $\supp(k_\sigma)\subseteq S$ such that
  \[
  \Quant_\alpha^+\bigl(X_n(W)\bigr)=W^{k_\sigma}
  \qquad\text{for every } W\in\sigma.
  \]
\end{enumerate}
Thus the upper quantile is either identically zero or a single monomial on each stratum.
\end{theorem}

\begin{proof}
Fix $W\in\sigma$. If $\supp(k)\not\subseteq S$, then $W^k=0$. Hence every positive value of $X_n(W)$ comes from a count vector whose support is contained in $S$, and the total mass of all positive values is exactly $\Pi_S$.

If $\Pi_S<\alpha$, then every threshold $t>0$ has
\[
\Prob(X_n(W)\ge t)\le \Pi_S<\alpha,
\]
whereas $\Prob(X_n(W)\ge 0)=1$. Therefore the upper $\alpha$-quantile is exactly $0$ on $\sigma$.

Assume now that $\Pi_S\ge \alpha$. For count vectors $k,\ell$ with supports contained in $S$, \cref{lem:ratio-linear} shows that the sign of $\log(W^k/W^\ell)$ is determined by a linear form whose sign is constant on the arrangement face $F$. Hence, as $W$ varies in $\sigma$, the weak ordering of the finite family $\{W^k:k\in\mathcal C_n\}$ is constant. Since the quantile index depends only on that weak ordering and on the fixed masses $\pi_k$, there exists a single count vector $k_\sigma$ with $\supp(k_\sigma)\subseteq S$ such that
\[
\Quant_\alpha^+\bigl(X_n(W)\bigr)=W^{k_\sigma}
\qquad (W\in\sigma).
\]
\end{proof}

\begin{remark}[Compatibility with the chamber theorem]
When $S=\{1,\dots,m\}$ and $F$ is a top-dimensional arrangement chamber, \cref{thm:stratum-quantile} reduces to \cref{thm:chamber-quantile}. The present refinement carries the same logic to all support faces and lower-dimensional arrangement faces.
\end{remark}

For a stratum $\sigma$, define its value by
\[
V(\sigma):=\sup_{W\in\sigma}\Quant_\alpha^+\bigl(X_n(W)\bigr).
\]

\begin{definition}[Child strata]
If $\sigma$ is a stratum, its child strata are those strata $\tau$ satisfying
\[
\tau\subseteq \partial_{\mathrm{rel}}\overline{\sigma},
\]
where $\overline{\sigma}$ denotes the closure of $\sigma$ in $\overline{\mathcal W}$ and $\partial_{\mathrm{rel}}$ denotes relative boundary in the affine span of $\overline{\sigma}$.
\end{definition}

\begin{lemma}[Boundary domination of the quantile]\label{lem:boundary-domination}
Fix a stratum $\sigma$ in the nonzero case of \cref{thm:stratum-quantile}, with active count vector $k_\sigma$. Let $W^{(j)}\in\sigma$ and suppose $W^{(j)}\to W^*\in\overline{\sigma}$. Then
\[
\Quant_\alpha^+\bigl(X_n(W^*)\bigr)\ge (W^*)^{k_\sigma}
=
\lim_{j\to\infty}\Quant_\alpha^+\bigl(X_n(W^{(j)})\bigr).
\]
\end{lemma}

\begin{proof}
Because the weak ordering is constant on $\sigma$, the set
\[
A_\sigma:=\{k\in\mathcal C_n: W^k\ge W^{k_\sigma}\text{ for }W\in\sigma\}
\]
is independent of the choice of $W\in\sigma$. By definition of $k_\sigma$,
\[
\sum_{k\in A_\sigma}\pi_k\ge \alpha.
\]
For each $k\in A_\sigma$ and each $j$, one has
\[
\bigl(W^{(j)}\bigr)^k\ge \bigl(W^{(j)}\bigr)^{k_\sigma}.
\]
Passing to the limit yields
\[
(W^*)^k\ge (W^*)^{k_\sigma}
\qquad (k\in A_\sigma).
\]
Therefore
\[
\Prob\!\left(X_n(W^*)\ge (W^*)^{k_\sigma}\right)
\ge
\sum_{k\in A_\sigma}\pi_k
\ge \alpha,
\]
which implies
\[
\Quant_\alpha^+\bigl(X_n(W^*)\bigr)\ge (W^*)^{k_\sigma}.
\]
The displayed limit identity follows from continuity of the monomial $W\mapsto W^{k_\sigma}$ on $\overline{\mathcal W}$.
\end{proof}

\begin{proposition}[Strict concavity on a fixed support face]\label{prop:strict-concavity}
Fix a stratum $\sigma=\sigma(S,F)$ in the nonzero case of \cref{thm:stratum-quantile}. In the ratio coordinates on $\mathcal W_S^\circ$, define
\begin{equation}\label{eq:psi-def}
\psi_\sigma(z)
:=
\log W(z)^{k_\sigma}
=
\sum_{i\in S\setminus\{r\}} (k_\sigma)_i z_i
-
n\log\Bigl(q_r+\sum_{i\in S\setminus\{r\}} q_i e^{z_i}\Bigr).
\end{equation}
Then $\psi_\sigma$ is strictly concave on $\R^{S\setminus\{r\}}$. Consequently,
\[
V(\sigma)=\sup_{z\in F} e^{\psi_\sigma(z)},
\]
and the same-support closure problem
\[
\sup_{z\in \overline F}\psi_\sigma(z)
\]
has at most one maximizer.
\end{proposition}

\begin{proof}
Using \eqref{eq:ratio-param}, one obtains the identity \eqref{eq:psi-def}. The first term is linear in $z$. The map
\[
z\mapsto \log\Bigl(q_r+\sum_{i\in S\setminus\{r\}} q_i e^{z_i}\Bigr)
\]
is a genuine log-sum-exp function with strictly positive coefficients and is therefore strictly convex on $\R^{S\setminus\{r\}}$. Hence $\psi_\sigma$ is strictly concave.

Since $F$ is a relatively open face of a finite hyperplane arrangement, its closure $\overline F$ is a closed polyhedron. The formula for $V(\sigma)$ follows from \cref{thm:stratum-quantile}, and strict concavity yields uniqueness of any maximizer on $\overline F$.
\end{proof}

\begin{remark}[Softmax and exponential-family viewpoint]
The function
\[
z\mapsto \log\Bigl(q_r+\sum_{i\in S\setminus\{r\}} q_i e^{z_i}\Bigr)
\]
is the log-partition function of a finite exponential family, so \cref{eq:psi-def} is the usual softmax log-likelihood in natural coordinates with empirical frequencies $k_\sigma/n$ and baseline prices $q$. The strict concavity in \cref{prop:strict-concavity} is the standard strict concavity of this natural-coordinate parametrization.
\end{remark}

\begin{proposition}[Boundary descent for a stratum]\label{prop:boundary-descent}
Fix a stratum $\sigma=\sigma(S,F)$ in the nonzero case of \cref{thm:stratum-quantile}. Then exactly one of the following alternatives holds.
\begin{enumerate}[label=(\roman*), itemsep=3pt]
  \item There exists a unique point $W_\sigma^*\in\sigma$ such that
  \[
  V(\sigma)=\Quant_\alpha^+\bigl(X_n(W_\sigma^*)\bigr)=\bigl(W_\sigma^*\bigr)^{k_\sigma}.
  \]
  \item Every maximizing sequence in $\sigma$ has a subsequence converging to a point in a finite union of child strata of $\sigma$.
\end{enumerate}
\end{proposition}

\begin{proof}
By \cref{prop:strict-concavity}, the problem $\sup_{z\in \overline F}\psi_\sigma(z)$ has at most one maximizer. If that maximizer exists and lies in $F$, then alternative (i) holds.

Assume that alternative (i) fails, and let $z^{(j)}\in F$ be a maximizing sequence for $\psi_\sigma$. Since
\[
0\le W_i\le \frac{1}{q_i}
\qquad (W\in\overline{\mathcal W}),
\]
the monomial $W^{k_\sigma}$ is uniformly bounded above on $\overline{\mathcal W}$, so such maximizing sequences have finite limiting value.

There are two cases.

\smallskip
\noindent\emph{Case 1: bounded ratio coordinates.}
If $\{z^{(j)}\}$ is bounded, then after passing to a subsequence one has $z^{(j)}\to z^*\in\overline F$. Since no maximizer lies in $F$, necessarily $z^*\in\overline F\setminus F$. The point $W(z^*)$ lies in the relative boundary of the same-support face, so it belongs to a stratum with the same support set $S$ but lower-dimensional arrangement face. Hence it lies in a child stratum of $\sigma$.

\smallskip
\noindent\emph{Case 2: unbounded ratio coordinates.}
If $\|z^{(j)}\|\to\infty$, then after passing to a subsequence some coordinate satisfies either $z_i^{(j)}\to+\infty$ or $z_i^{(j)}\to-\infty$. If $z_i^{(j)}\to+\infty$, then the denominator in \eqref{eq:ratio-param} tends to $+\infty$, so $W_r(z^{(j)})\to 0$. If $z_i^{(j)}\to-\infty$, then
\[
W_i(z^{(j)})=e^{z_i^{(j)}}W_r(z^{(j)})\to 0
\]
because $W_r(z^{(j)})\le 1/q_r$. In either situation, at least one coordinate from the support set $S$ tends to $0$.

Since the closed support face $\overline{\mathcal W}_S$ is compact, the corresponding wealth profiles $W(z^{(j)})$ admit a convergent subsequence with limit $W^*\in\overline{\mathcal W}_S$. The calculation above shows that $W^*$ belongs to a proper support face $\overline{\mathcal W}_T$ for some strict subset $T\subsetneq S$. Hence $W^*$ lies in a child stratum of $\sigma$ with smaller support.

In both cases, every maximizing sequence has a subsequence converging to a point in a finite union of child strata.
\end{proof}

\begin{corollary}[Child-stratum domination]\label{cor:child-domination}
Fix a stratum $\sigma$.
\begin{enumerate}[label=(\alph*), itemsep=3pt]
  \item If $\sigma$ is in the zero case of \cref{thm:stratum-quantile}, then $V(\sigma)=0$.
  \item If $\sigma$ is in the nonzero case of \cref{thm:stratum-quantile} and alternative \emph{(ii)} of \cref{prop:boundary-descent} holds, then
  \[
  V(\sigma)\le \max_{\tau\prec\sigma} V(\tau),
  \]
  where $\tau\prec\sigma$ ranges over the child strata of $\sigma$.
\end{enumerate}
\end{corollary}

\begin{proof}
Part (a) is immediate from \cref{thm:stratum-quantile}. For part (b), let $W^{(j)}\in\sigma$ be a maximizing sequence. By \cref{prop:boundary-descent}, after passing to a subsequence one has $W^{(j)}\to W^*$ for some point $W^*$ belonging to a child stratum $\tau$ of $\sigma$. By \cref{lem:boundary-domination},
\[
\lim_{j\to\infty}\Quant_\alpha^+\bigl(X_n(W^{(j)})\bigr)
\le
\Quant_\alpha^+\bigl(X_n(W^*)\bigr)
\le
V(\tau).
\]
Taking the supremum over all child strata gives the claim.
\end{proof}

\begin{theorem}[Weak exact recursive boundary algorithm]\label{thm:weak-recursion}
Fix $(p,q,n,\alpha)$. Then the finite-horizon upper-quantile Kelly problem
\[
\sup_{W\in\overline{\mathcal W}} \Quant_\alpha^+\bigl(X_n(W)\bigr)
\]
admits a finite exact recursive solution on the stratification of \cref{lem:finite-stratification}.

More precisely:
\begin{enumerate}[label=(\alph*), itemsep=4pt]
  \item Each stratum $\sigma$ is either a zero stratum with $V(\sigma)=0$, or a nonzero stratum on which the quantile is a single monomial $W^{k_\sigma}$.
  \item On every nonzero stratum, the same-support optimization problem is a strictly concave finite-dimensional problem in ratio coordinates.
  \item If a nonzero stratum has no interior maximizer, then its value is dominated by the values of finitely many child strata.
  \item If one orders strata by the lexicographic rank
  \[
  \rho(\sigma(S,F)):=\bigl(|S|,\dim F\bigr),
  \]
  then every child stratum has strictly smaller rank. Consequently the recursive descent terminates after finitely many steps.
\end{enumerate}
In particular, processing all strata in decreasing rank yields a finite exact algorithm for the global problem. One may also view the same procedure as a top-down boundary descent that begins with the full-support chambers and passes to child strata whenever needed. In particular, the global supremum is attained.
\end{theorem}

\begin{proof}
Parts (a)--(c) are \cref{thm:stratum-quantile,prop:strict-concavity,cor:child-domination}. It remains to prove the rank drop in part (d).

Let $\tau$ be a child stratum of $\sigma=\sigma(S,F)$. If $\tau$ has the same support set $S$, then by definition it lies in the relative boundary of $\overline F$ inside $\R^{S\setminus\{r\}}$. Hence its arrangement face is a proper face of $\overline F$, so
\[
\dim F_\tau<\dim F,
\qquad
|S_\tau|=|S|.
\]
Thus $\rho(\tau)<\rho(\sigma)$.

If instead $\tau$ lies on a smaller support face, then its support set satisfies $S_\tau\subsetneq S$, and therefore
\[
|S_\tau|<|S|.
\]
Again $\rho(\tau)<\rho(\sigma)$.

Because only finitely many strata exist, the lexicographically ordered rank set is finite. Therefore no infinite descent is possible. Exactness follows because the strata partition $\overline{\mathcal W}$, every nonzero stratum is handled either directly by its unique interior maximizer or indirectly by passing to child strata, and zero strata contribute nothing. The same rank induction shows that the exact value selected by the recursion is attained at some terminal descendant stratum, and therefore the global supremum over $\overline{\mathcal W}$ is attained.
\end{proof}

\begin{remark}[Complexity of the weak solver]\label{rem:complexity}
The weakness of \cref{thm:weak-recursion} is computational rather than mathematical. For a support set $S$ of size $s$, the number of relevant count vectors is
\[
M_s=\binom{n+s-1}{s-1},
\]
so the arrangement $\mathcal H_S$ has at most $\binom{M_s}{2}$ hyperplanes in $\R^{s-1}$. Standard hyperplane-arrangement bounds therefore give at most
\[
O\!\left(\binom{M_s}{2}^{\,s-1}\right)
\]
faces on that support face; see, for example, Zaslavsky \cite{Zaslavsky1975}. Summing over support sets yields a very coarse overall bound of order
\[
O\!\left(
2^m \binom{n+m-1}{m-1}^{\,2(m-1)}
\right).
\]
This is perfectly adequate for the present existential theorem, but far too crude to count as an efficient algorithm.
\end{remark}

\begin{remark}[A concise slogan]
\Cref{thm:main-structural,thm:weak-recursion} may be summarized as follows: finite-horizon quantile Kelly is \emph{stratified shadow-Kelly with finite boundary descent}. The chamber theorem identifies the active shadow law on each interior chamber, while the support-face refinement turns boundary behavior into a finite recursive problem.
\end{remark}

\begin{example}[A worked ternary boundary descent]\label{ex:ternary-main}
Let
\[
p=(0.6,0.3,0.1),
\qquad
q=\Bigl(\frac13,\frac13,\frac13\Bigr),
\]
so the one-period wealth simplex is
\[
\overline{\mathcal W}=\{W\in[0,\infty)^3: W_1+W_2+W_3=3\}.
\]
Take horizon $n=2$ and upper median level $\alpha=\tfrac12$. The six count vectors in $\mathcal C_2$, together with the corresponding monomials and probabilities, are
\[
\begin{array}{c|c|c}
k & W^k & \pi_k \\ \hline
(2,0,0) & W_1^2 & 0.36 \\
(1,1,0) & W_1W_2 & 0.36 \\
(1,0,1) & W_1W_3 & 0.12 \\
(0,2,0) & W_2^2 & 0.09 \\
(0,1,1) & W_2W_3 & 0.06 \\
(0,0,2) & W_3^2 & 0.01
\end{array}
\]
Consider the full-support stratum cut out by
\[
W_1>W_2>W_3
\qquad\text{and}\qquad
W_2^2>W_1W_3.
\]
This extra inequality selects one of the two top-dimensional strata inside the sector $W_1>W_2>W_3$; the appendix records the companion stratum $W_1W_3>W_2^2$ and shows that both yield the same active count vector. Inside this stratum the monomials are ordered
\[
W_1^2>W_1W_2>W_2^2>W_1W_3>W_2W_3>W_3^2.
\]
Thus the upper median is reached at the second term, since
\[
0.36<\tfrac12<0.36+0.36=0.72.
\]
Hence the active count vector is
\[
k_\sigma=(1,1,0),
\qquad
\Med^+\bigl(X_2(W)\bigr)=W_1W_2
\quad (W\in\sigma).
\]
The corresponding shadow-Kelly point is
\[
W^{(k_\sigma)}=\frac{k_\sigma}{2q}=(1.5,1.5,0),
\]
which does not lie in the chosen full-support stratum. By \cref{thm:weak-recursion}, the stratum optimum must therefore descend to the relative boundary.

The first descent is to the support face $W_3=0$, where $W_1+W_2=3$. On the relative interior with $W_1>W_2>0$, the positive monomials are ordered
\[
W_1^2>W_1W_2>W_2^2>0,
\]
so the upper median remains $W_1W_2$. The reduced subproblem is therefore
\[
\max\{W_1W_2: W_1>0,\ W_2>0,\ W_1+W_2=3\},
\]
whose unique maximizer on the closed support face is
\[
W_1=W_2=1.5.
\]
Thus the recursion terminates on the tie stratum
\[
W_1=W_2=1.5,
\qquad
W_3=0,
\]
where
\[
W_1^2=W_1W_2=W_2^2=2.25.
\]
Hence the exact optimal upper median along this descent is
\[
\Med^+\bigl(X_2(W)\bigr)=2.25.
\]
For comparison, the ordinary Kelly point is
\[
W^{\mathrm K}=\frac{p}{q}=(1.8,0.9,0.3),
\]
which lies in the original full-support stratum. Thus even in a very small ternary problem the exact finite-horizon median optimizer need not coincide with the asymptotic Kelly profile. A more detailed geometric analysis of this example, including the refined subdivision of the sector $W_1>W_2>W_3$, is given in \cref{app:ternary-example}.
\end{example}

\section{Asymptotic collapse to ordinary Kelly}

Define the one-period Kelly functional on the closed simplex by
\[
L(W):=
\begin{cases}
\sum_{i=1}^m p_i \log W_i, & W\in \mathcal W^{++},\\
-\infty, & W\in \overline{\mathcal W}\setminus \mathcal W^{++}.
\end{cases}
\]
Also define the scaled log-quantile by
\[
G_{n,\alpha}(W):=
\begin{cases}
\dfrac1n \log \Quant_\alpha^+\bigl(X_n(W)\bigr), & \Quant_\alpha^+\bigl(X_n(W)\bigr)>0,\\
-\infty, & \Quant_\alpha^+\bigl(X_n(W)\bigr)=0.
\end{cases}
\]
Thus the finite-horizon optimization problem may be written as
\[
\sup_{W\in\overline{\mathcal W}} G_{n,\alpha}(W).
\]

\begin{lemma}[Ordinary Kelly functional and its unique maximizer]\label{lem:kelly-functional}
The functional $L$ is upper semicontinuous on $\overline{\mathcal W}$ and has the unique maximizer
\[
W^{\mathrm K}_i=\frac{p_i}{q_i},
\qquad i=1,\dots,m.
\]
Its maximal value is
\[
L^*:=L(W^{\mathrm K})=\sum_{i=1}^m p_i \log \frac{p_i}{q_i}.
\]
Moreover, if $W\to \partial \overline{\mathcal W}$, then $L(W)\to -\infty$.
\end{lemma}

\begin{proof}
If $W\in \mathcal W^{++}$ and $r_i:=q_iW_i$, then $r=(r_1,\dots,r_m)$ is a probability vector because $q\cdot W=1$. Hence
\[
L(W)=\sum_{i=1}^m p_i \log \frac{r_i}{q_i}
=\sum_{i=1}^m p_i \log \frac{p_i}{q_i}-\sum_{i=1}^m p_i \log \frac{p_i}{r_i}
=L^*-D_{\mathrm{KL}}(p\|r).
\]
Therefore $L(W)\le L^*$, with equality if and only if $r=p$, equivalently $W_i=p_i/q_i$ for every $i$. This proves the maximizing statement.

Upper semicontinuity on $\overline{\mathcal W}$ is immediate from continuity on $\mathcal W^{++}$ and the convention $L=-\infty$ on the boundary. Since every $p_i>0$ and $0\le W_i\le 1/q_i$ on $\overline{\mathcal W}$, whenever $W\to \partial \overline{\mathcal W}$ at least one coordinate tends to $0$, so the corresponding term $p_i\log W_i$ tends to $-\infty$ while the remaining terms stay bounded above. Thus $L(W)\to -\infty$ at the boundary.
\end{proof}

\begin{proposition}[Uniform interior convergence of scaled log-quantiles]\label{prop:uniform-interior}
Let $K\subset \mathcal W^{++}$ be compact and fix $\alpha\in(0,1)$. Then
\[
\sup_{W\in K} \bigl|G_{n,\alpha}(W)-L(W)\bigr|\to 0
\qquad\text{as } n\to\infty.
\]
\end{proposition}

\begin{proof}
For $W\in K$, write
\[
A_n(W):=\frac1n \log X_n(W)=\frac1n\sum_{t=1}^n Y_t(W),
\]
where $Y_t(W)=\log W_{I_t}$ and $I_t$ are i.i.d. with law $p$. Then $\E[Y_t(W)]=L(W)$. Since $K$ is compact in $\mathcal W^{++}$, there exists $B_K<\infty$ such that
\[
|\log W_i|\le B_K
\qquad (W\in K,\ i=1,\dots,m).
\]
Therefore each $Y_t(W)$ takes values in $[-B_K,B_K]$, uniformly over $W\in K$. If $B_K=0$, then every $Y_t(W)$ is identically $0$ on $K$ and the conclusion is immediate. Thus we may assume $B_K>0$. Hoeffding's inequality \cite{Hoeffding1963} gives
\begin{equation}\label{eq:uniform-hoeffding}
\Prob\!\left(\bigl|A_n(W)-L(W)\bigr|\ge \varepsilon\right)
\le 2\exp\!\left(-\frac{n\varepsilon^2}{2B_K^2}\right)
\qquad (W\in K).
\end{equation}

Let
\[
\widetilde Q_{n,\alpha}(W):=\sup\{t\in\R: \Prob(A_n(W)\ge t)\ge \alpha\}.
\]
Since the exponential map is increasing,
\[
G_{n,\alpha}(W)=\widetilde Q_{n,\alpha}(W)
\qquad (W\in K).
\]
Fix $\varepsilon>0$. By \eqref{eq:uniform-hoeffding}, for every $W\in K$,
\[
\Prob\bigl(A_n(W)\ge L(W)+\varepsilon\bigr)
\le \exp\!\left(-\frac{n\varepsilon^2}{2B_K^2}\right),
\]
and
\[
\Prob\bigl(A_n(W)\ge L(W)-\varepsilon\bigr)
\ge 1-\exp\!\left(-\frac{n\varepsilon^2}{2B_K^2}\right).
\]
For all sufficiently large $n$, the exponential bound is smaller than $\min\{\alpha,1-\alpha\}$. Then uniformly in $W\in K$ one has
\[
\widetilde Q_{n,\alpha}(W)\le L(W)+\varepsilon
\quad\text{and}\quad
\widetilde Q_{n,\alpha}(W)\ge L(W)-\varepsilon.
\]
Hence
\[
\sup_{W\in K} \bigl|G_{n,\alpha}(W)-L(W)\bigr|\le \varepsilon
\]
for all large $n$, which proves the claim.
\end{proof}

\begin{lemma}[Uniform exclusion of the boundary region]\label{lem:boundary-exclusion}
Let
\[
M:=\max_{1\le i\le m} \log \frac1{q_i}.
\]
For $\eta>0$, define
\[
K_\eta:=\{W\in \overline{\mathcal W}: W_i\ge \eta \text{ for all } i\}
\]
and
\[
B(\eta):=\max_{1\le i\le m}\left[\Bigl(1-\frac{p_i}{2}\Bigr)M+\frac{p_i}{2}\log \eta\right].
\]
Then $B(\eta)\to -\infty$ as $\eta\downarrow 0$, and for every fixed $\eta>0$ there exists $n_0(\eta,\alpha)$ such that
\[
G_{n,\alpha}(W)\le B(\eta)
\qquad\text{for every } W\in \overline{\mathcal W}\setminus K_\eta
\]
whenever $n\ge n_0(\eta,\alpha)$.
\end{lemma}

\begin{proof}
The limit $B(\eta)\to -\infty$ is immediate because every $p_i>0$ and $\log \eta\to -\infty$. Fix $\eta>0$ and $W\notin K_\eta$. Then $W_i\le \eta$ for at least one index $i$. If some coordinate of $W$ vanishes, we interpret $A_n(W)=n^{-1}\log X_n(W)$ pathwise, so $A_n(W)=-\infty$ whenever an outcome with zero wealth occurs at least once; this can only strengthen the desired upper bound. Thus it suffices to prove the estimate in the case the displayed logarithmic sum is finite, and on that event one has
\[
A_n(W)=\sum_{j=1}^m \frac{N_j}{n}\log W_j.
\]
On the event
\[
E_i:=\left\{\frac{N_i}{n}\ge \frac{p_i}{2}\right\},
\]
one has
\[
A_n(W)
\le \frac{N_i}{n}\log \eta+\Bigl(1-\frac{N_i}{n}\Bigr)M
\le \frac{p_i}{2}\log \eta+\Bigl(1-\frac{p_i}{2}\Bigr)M
\le B(\eta).
\]
Here we used $W_j\le 1/q_j$ for every $j$, hence $\log W_j\le M$ whenever the logarithm is finite.

Now $N_i\sim \mathrm{Binomial}(n,p_i)$, so Hoeffding's inequality yields
\[
\Prob(E_i^c)
=\Prob\!\left(\frac{N_i}{n}<\frac{p_i}{2}\right)
\le \exp\!\left(-\frac{n p_i^2}{2}\right)
\le \exp\!\left(-\frac{n p_*^2}{2}\right),
\]
where $p_*:=\min_i p_i>0$. For all sufficiently large $n$, this bound is smaller than $\alpha$. Therefore
\[
\Prob\bigl(A_n(W)>B(\eta)\bigr)<\alpha
\]
for every $W\notin K_\eta$. By the definition of the upper $\alpha$-quantile, this implies
\[
G_{n,\alpha}(W)=\widetilde Q_{n,\alpha}(W)\le B(\eta).
\]
This is uniform over $W\in \overline{\mathcal W}\setminus K_\eta$.
\end{proof}

\begin{theorem}[Asymptotic value convergence and collapse of exact maximizers]
\label{thm:asymptotic-collapse}
Fix $\alpha\in(0,1)$. Then
\begin{enumerate}[label=(\alph*), itemsep=4pt]
  \item The optimal scaled log-quantile converges to the ordinary Kelly value:
  \[
  \sup_{W\in\overline{\mathcal W}} G_{n,\alpha}(W)\to L^*=L(W^{\mathrm K}).
  \]
  Equivalently,
  \[
  \frac1n \log \sup_{W\in\overline{\mathcal W}} \Quant_\alpha^+\bigl(X_n(W)\bigr)
  \to \sum_{i=1}^m p_i \log \frac{p_i}{q_i}.
  \]
  \item If $W_n^{(\alpha)}$ is any exact maximizer of the finite-horizon upper-quantile problem \eqref{eq:master-problem}, then
  \[
  W_n^{(\alpha)}\to W^{\mathrm K}.
  \]
\end{enumerate}
\end{theorem}

\begin{proof}
Let $\varepsilon>0$. Choose $\eta>0$ so small that $W^{\mathrm K}\in K_\eta$ and
\begin{equation}\label{eq:eta-choice-value}
B(\eta)<L^*-2\varepsilon.
\end{equation}
By \cref{prop:uniform-interior},
\begin{equation}\label{eq:uniform-K-eta}
\sup_{W\in K_\eta} |G_{n,\alpha}(W)-L(W)|<\varepsilon
\end{equation}
for all sufficiently large $n$. By \cref{lem:boundary-exclusion},
\begin{equation}\label{eq:boundary-upper}
\sup_{W\notin K_\eta} G_{n,\alpha}(W)\le B(\eta)<L^*-2\varepsilon
\end{equation}
for all sufficiently large $n$. Since $W^{\mathrm K}\in K_\eta$, \eqref{eq:uniform-K-eta} gives
\[
G_{n,\alpha}(W^{\mathrm K})>L^*-\varepsilon.
\]
Combining this with \eqref{eq:boundary-upper}, we see that for all large $n$ the global supremum is attained inside $K_\eta$, and
\[
L^*-\varepsilon
\le \sup_{W\in\overline{\mathcal W}} G_{n,\alpha}(W)
= \sup_{W\in K_\eta} G_{n,\alpha}(W)
\le \sup_{W\in K_\eta} L(W)+\varepsilon
\le L^*+\varepsilon.
\]
This proves part (a).

For part (b), let $U$ be any open neighborhood of $W^{\mathrm K}$. Choose $\eta>0$ so that $W^{\mathrm K}\in K_\eta$. Since $K_\eta\setminus U$ is compact and does not contain the unique maximizer $W^{\mathrm K}$ of $L$, there exists $\delta>0$ such that
\begin{equation}\label{eq:gap-away-U}
\sup_{W\in K_\eta\setminus U} L(W)\le L^*-3\delta.
\end{equation}
After possibly shrinking $\eta$ further, we may also assume
\begin{equation}\label{eq:boundary-gap-U}
W^{\mathrm K}\in K_\eta
\quad\text{and}\quad
B(\eta)<L^*-3\delta.
\end{equation}
For all sufficiently large $n$, \cref{prop:uniform-interior,lem:boundary-exclusion} imply
\[
\sup_{W\in K_\eta\setminus U} G_{n,\alpha}(W)<L^*-2\delta
\quad\text{and}\quad
\sup_{W\notin K_\eta} G_{n,\alpha}(W)<L^*-3\delta,
\]
while
\[
G_{n,\alpha}(W^{\mathrm K})>L^*-\delta.
\]
Therefore every exact maximizer $W_n^{(\alpha)}$ lies in $U$ for all large $n$. Since $U$ was arbitrary, $W_n^{(\alpha)}\to W^{\mathrm K}$.
\end{proof}

\begin{corollary}[Asymptotically optimal sequences collapse to Kelly]\label{cor:asymptotically-optimal}
Fix $\alpha\in(0,1)$, and let $W_n\in\overline{\mathcal W}$ be any sequence such that
\[
G_{n,\alpha}(W_n)=\sup_{W\in\overline{\mathcal W}} G_{n,\alpha}(W)+o(1).
\]
Then $W_n\to W^{\mathrm K}$. In particular, ordinary Kelly is asymptotically optimal for every fixed upper quantile:
\[
G_{n,\alpha}(W^{\mathrm K})=\sup_{W\in\overline{\mathcal W}} G_{n,\alpha}(W)+o(1).
\]
\end{corollary}

\begin{proof}
The final statement follows immediately from \cref{thm:asymptotic-collapse}(a) and the pointwise convergence at $W^{\mathrm K}$ furnished by \cref{prop:uniform-interior}.

For the convergence claim, let $U$ be any open neighborhood of $W^{\mathrm K}$. Choose $\eta>0$ and $\delta>0$ so that \eqref{eq:gap-away-U} and \eqref{eq:boundary-gap-U} hold. Then for all sufficiently large $n$,
\[
\sup_{W\notin U} G_{n,\alpha}(W)\le L^*-2\delta
\]
by the same argument used in the proof of \cref{thm:asymptotic-collapse}(b), while \cref{thm:asymptotic-collapse}(a) gives
\[
\sup_{W\in\overline{\mathcal W}} G_{n,\alpha}(W)\ge L^*-\delta.
\]
Hence
\[
\sup_{W\notin U} G_{n,\alpha}(W)
\le \sup_{W\in\overline{\mathcal W}} G_{n,\alpha}(W)-\delta
\]
for all large $n$. This contradicts the assumed asymptotic optimality unless $W_n\in U$ eventually. Since $U$ was arbitrary, $W_n\to W^{\mathrm K}$.

\end{proof}

\begin{remark}[First-order asymptotic scope]
The asymptotic results above are intentionally first-order. They identify the limiting optimal growth exponent and the collapse of near-optimal wealth profiles to the ordinary Kelly point, but they do not attempt any local central-limit, Berry--Esseen, or $n^{-1/2}$-scale optimizer expansion. Those genuinely stronger questions lie beyond the scope of the present exact finite-horizon note.
\end{remark}

\begin{remark}[Toward higher-order finite-horizon asymptotics]
A natural next step is to supplement \cref{thm:asymptotic-collapse} with a local correction theory describing how the exact finite-horizon optimizer and value deviate from ordinary Kelly at large but finite $n$. In the parlay setting, \cite{Long2026Parlay} develops explicit perturbative asymptotics in wealth-profile coordinates, including the Whitrow regime. That viewpoint suggests that the present multinomial quantile problem should also admit a higher-order expansion once the local chamber/stratum geometry near $W^{\mathrm K}$ is understood.
\end{remark}

\section{Further remarks and conjectures}

\begin{remark}[All upper quantiles, not only the median]
The exact chamber theorem, the weak recursion theorem, and the first-order asymptotic collapse theorem are more general than the median: they hold verbatim for every fixed upper quantile level $\alpha\in(0,1)$. The only change is the active quantile index on each chamber or stratum.
\end{remark}

\begin{remark}[Simultaneous wagers with a fixed finite joint state space]
Suppose one period consists of any fixed finite joint state space $\Omega$, and suppose a one-period decision determines a feasible wealth family $\mathcal W_{\mathrm{feas}}\subseteq\overline{\mathcal W}$ on that state space. Then \cref{prop:restricted-family} applies without change. In particular, repeated finite-horizon quantile optimization for simultaneous wagers is exactly the restriction of the same chamber decomposition to the feasible one-period wealth family. This is especially natural when simultaneous wager menus are further restricted by drawdown-style admissibility conditions: in that case the local subproblems are again risk-constrained shadow-Kelly problems in the sense of \cite{Long2026Risk}, now embedded inside the finite-horizon chamber/stratum decomposition.
\end{remark}

\begin{remark}[Where the theorem uses repeated identical events]
The multinomial reduction is exact because the same one-period event, with the same outcome law and the same one-period wealth profile, is repeated independently through time. If the event changes from period to period, then the exact finite-horizon problem lives on the full path space rather than a single multinomial count simplex. The same geometric idea survives, but the clean count-vector formulation is replaced by a path-space chamber decomposition.
\end{remark}

\begin{remark}[Framework interpretation]
The note is a direct application of the same framework machine developed in \cite{Long2026Single,Long2026Simultaneous,Long2026Parlay,Long2026Risk}.
\begin{enumerate}[label=(\roman*), itemsep=2pt]
  \item \textbf{Wealth profile.} The primitive variable is the one-period state-contingent wealth profile $W$.
  \item \textbf{Arrow--Debreu pricing.} Feasibility is expressed by $q\cdot W=1$.
  \item \textbf{Shadow Kelly.} Once a chamber or stratum count vector $k$ is fixed, the local finite-horizon problem is a one-period Kelly problem for the shadow empirical law $k/n$.
  \item \textbf{Finite boundary descent.} When a local shadow-Kelly optimizer leaves its stratum, the residual problem passes to lower-rank strata of the same finite geometry.
\end{enumerate}
This is why the finite-horizon theorem belongs to the same theory as the one-period Kelly, simultaneous-decoupling, parlay-factorization, and risk-constrained support-geometry results developed in \cite{Long2026Single,Long2026Simultaneous,Long2026Parlay,Long2026Risk}, rather than standing as an isolated result.
\end{remark}

\begin{conjecture}[Strong pruning for exact computation]\label{conj:strong-pruning}
There is a sharper exact recursive algorithm which, given $(p,q,n,\alpha)$, traverses only a properly pruned subset of boundary strata singled out by local violation data, rather than the full weak stratified recursion of \cref{thm:weak-recursion}. In particular, exact computation should be possible without anything close to exhaustive exploration of descendant strata.
\end{conjecture}

\begin{conjecture}[Higher-order finite-horizon correction to Kelly]\label{conj:higher-order}
Fix $\alpha\in(0,1)$ and suppose the local chamber/stratum geometry is nondegenerate near $W^{\mathrm K}$. Then the exact finite-horizon maximizers $W_n^{(\alpha)}$ and the optimal values $\sup_{W\in\overline{\mathcal W}} G_{n,\alpha}(W)$ admit a higher-order asymptotic expansion about ordinary Kelly, with a leading correction determined by the local multinomial quantile geometry. In particular, one expects a finite-$n$ correction law for $W_n^{(\alpha)}-W^{\mathrm K}$, analogous in spirit to the perturbative Whitrow-type expansions obtained for parlays in \cite{Long2026Parlay}.
\end{conjecture}

\begin{remark}[Dual viewpoint]
The weak recursion theorem leaves open where efficient pruning truly lives. A natural next step is a dual reformulation in terms of Bregman or entropic projection, where one might hope to identify the correct boundary branches directly rather than by weak finite descent.
\end{remark}

\begin{remark}[What has been proved and what remains open]
The exact chamber theorem, the exact shadow-Kelly reduction, the weak exact recursive boundary algorithm, and the first-order asymptotic collapse to ordinary Kelly are fully proved above. The stronger pruning problem remains open here.
\end{remark}

\section*{Acknowledgment of scope}
The purpose of this note is to isolate the exact repeated multi-outcome finite-horizon theorem in a self-contained form. It is not intended as a complete treatment of dynamic or path-dependent menus, nor as a full higher-order asymptotic theory. Rather, it establishes the exact structural bridge from finite-horizon quantiles to chamberwise shadow-Kelly optimization, upgrades that bridge to a weak exact recursive solver on the full closed simplex, and proves the natural first-order asymptotic collapse back to ordinary Kelly.

\appendix

\section{Expanded binary example}\label{app:binary-example}
We revisit \cref{ex:binary-main} in a slightly more systematic form. The data are
\[
p=(0.6,0.4),\qquad q=(1,1),\qquad n=3,
\]
so $\overline{\mathcal W}=\{(U,D)\in[0,\infty)^2:U+D=1\}$. The possible count vectors, monomials, and probabilities are
\[
\begin{array}{c|c|c}
k & W^k & \pi_k \\ \hline
(3,0) & U^3 & 0.216 \\
(2,1) & U^2D & 0.432 \\
(1,2) & UD^2 & 0.288 \\
(0,3) & D^3 & 0.064
\end{array}
\]
There are exactly two open chambers, $\Gamma_+=\{U>D\}$ and $\Gamma_-=\{U<D\}$, separated by the wall $U=D$.

On $\Gamma_+$ the monomials are ordered
\[
U^3>U^2D>UD^2>D^3,
\]
whereas on $\Gamma_-$ they are ordered
\[
D^3>UD^2>U^2D>U^3.
\]
For the upper median, both chambers produce the same active count vector $(2,1)$ because
\[
0.216<\tfrac12<0.216+0.432=0.648
\]
in the first ordering and
\[
0.064+0.288=0.352<\tfrac12<0.352+0.432=0.784
\]
in the second. Thus
\[
\Med^+\bigl(X_3(U,D)\bigr)=U^2D
\]
on both open chambers. The associated shadow law is $\widehat p=(2/3,1/3)$, and the corresponding shadow-Kelly point is
\[
\Bigl(\frac23,\frac13\Bigr).
\]
This point lies in $\Gamma_+$ and solves the $\Gamma_+$ subproblem exactly, while on $\Gamma_-$ it lies outside the chamber and forces the chamber optimum to the boundary wall. This yields the two chamber values recorded in \cref{ex:binary-main}.

It is also instructive to look briefly at a higher quantile level, say $\alpha=0.8$. On $\Gamma_+$ one finds
\[
0.216+0.432=0.648<0.8<0.936=0.216+0.432+0.288,
\]
so the active count vector is $(1,2)$ and the local quantile is $UD^2$. Its shadow-Kelly point is $(1/3,2/3)$, which lies outside $\Gamma_+$, so the chamber optimum again collapses to the wall $U=D$. On $\Gamma_-$ one has
\[
0.064+0.288+0.432=0.784<0.8<1,
\]
so the active count vector is $(3,0)$ and the local quantile is $U^3$. Its shadow-Kelly point is $(1,0)$, which also lies outside $\Gamma_-$. Thus at $\alpha=0.8$ both chambers are boundary-driven, and the global optimum is attained on the wall $U=D=1/2$ with value $1/8$. This small calculation illustrates that even in the binary case the active monomial can depend strongly on the quantile level.

\section{Expanded ternary example}\label{app:ternary-example}
We now expand \cref{ex:ternary-main}. The data are
\[
p=(0.6,0.3,0.1),
\qquad
q=\Bigl(\frac13,\frac13,\frac13\Bigr),
\qquad
n=2,
\qquad
\alpha=\tfrac12,
\]
so $\overline{\mathcal W}=\{W\in[0,\infty)^3:W_1+W_2+W_3=3\}$. The six count vectors, monomials, and probabilities are
\[
\begin{array}{c|c|c}
k & W^k & \pi_k \\ \hline
(2,0,0) & W_1^2 & 0.36 \\
(1,1,0) & W_1W_2 & 0.36 \\
(1,0,1) & W_1W_3 & 0.12 \\
(0,2,0) & W_2^2 & 0.09 \\
(0,1,1) & W_2W_3 & 0.06 \\
(0,0,2) & W_3^2 & 0.01
\end{array}
\]

The sector $W_1>W_2>W_3$ is already subdivided by the extra tie hyperplane
\[
W_2^2=W_1W_3.
\]
Indeed, inside this sector there are two top-dimensional strata:
\begin{align*}
\sigma_A &: W_1>W_2>W_3,\quad W_2^2>W_1W_3,\\
\sigma_B &: W_1>W_2>W_3,\quad W_1W_3>W_2^2.
\end{align*}
On these two strata the monomial orderings are
\[
W_1^2>W_1W_2>W_2^2>W_1W_3>W_2W_3>W_3^2
\qquad (W\in\sigma_A),
\]
and
\[
W_1^2>W_1W_2>W_1W_3>W_2^2>W_2W_3>W_3^2
\qquad (W\in\sigma_B).
\]
In both cases the first two probabilities are $0.36$ and $0.36$, so the upper median is already reached at the second term. Therefore the active count vector on both $\sigma_A$ and $\sigma_B$ is
\[
k_\sigma=(1,1,0),
\qquad
\Med^+\bigl(X_2(W)\bigr)=W_1W_2.
\]
The shadow-Kelly point is the same on both strata:
\[
W^{(k_\sigma)}=(1.5,1.5,0).
\]
This point lies on neither full-support stratum, so both descend to the support face $W_3=0$.

On the open support face $W_3=0$ with $W_1>W_2>0$, the only positive monomials are
\[
W_1^2>W_1W_2>W_2^2>0.
\]
Their cumulative probabilities are
\[
0.36,
\qquad
0.72,
\qquad
0.81,
\]
so the upper median again equals $W_1W_2$. The reduced problem is the one-variable concave maximization
\[
\max\{W_1W_2: W_1+W_2=3,\ W_1,W_2>0\},
\]
whose optimizer on the closed face is the tie point
\[
W_1=W_2=1.5,
\qquad
W_3=0.
\]
At this point,
\[
W_1^2=W_1W_2=W_2^2=2.25,
\]
and the total mass of the tied positive values is
\[
0.36+0.36+0.09=0.81>\tfrac12.
\]
Hence the upper median on the terminal tie stratum is exactly $2.25$.

This example also exhibits the zero-stratum phenomenon from \cref{thm:stratum-quantile}. On the one-point support face
\[
W_1=3,
\qquad
W_2=W_3=0,
\]
the only positive terminal wealth occurs for count vector $(2,0,0)$, whose probability is $0.36<\tfrac12$. Therefore the upper median on that face is exactly $0$. Similar remarks apply to the one-point faces concentrated on outcomes $2$ and $3$.

Finally, the ordinary Kelly point is
\[
W^{\mathrm K}=\frac{p}{q}=(1.8,0.9,0.3).
\]
It lies in the stratum $\sigma_A$, since
\[
1.8>0.9>0.3
\qquad\text{and}\qquad
0.9^2=0.81>0.54=(1.8)(0.3).
\]
Thus the finite-horizon median optimizer produced by the recursive descent is visibly different from the asymptotic Kelly point even though both live naturally inside the same simplex geometry. The example therefore illustrates, in the smallest genuinely multi-outcome setting, the exact chamber logic, the support-face refinement, the boundary recursion, and the first-order collapse theme of the main text.

\end{document}